\documentclass[reqno,11pt]{amsart}
\usepackage{amsmath, latexsym, amsfonts, amssymb}
\usepackage{graphics,epsf,psfrag}

\usepackage{stmaryrd}

\numberwithin{equation}{section}

   \newtheorem{theorema}{Theorem}
   
   \newtheorem{theorem}{Theorem}[section]
   
    \newtheorem{lemma}[theorem]{Lemma}
   \newtheorem{proposition}[theorem]{Proposition}
   \newtheorem{rem}[theorem]{Remark}

\renewcommand{\tilde}{\widetilde}

\newcommand{\tf}{\textsc{f}}
\newcommand{\cc}{\complement}

\newcommand{\lint}{\llbracket}
\newcommand{\rint}{\rrbracket}

\newcommand{\cA}{\ensuremath{\mathcal A}} 
\newcommand{\cB}{\ensuremath{\mathcal B}}

\newcommand{\cI}{\ensuremath{\mathcal I}}

\newcommand{\bbE}{{\ensuremath{\mathbb E}} }

\newcommand{\bbN}{{\ensuremath{\mathbb N}} } 
 
\newcommand{\bbP}{{\ensuremath{\mathbb P}} } 
 
\newcommand{\bbR}{{\ensuremath{\mathbb R}} }

\newcommand{\var}{{\rm Var} }

\newcommand{\ga}{\alpha}
\newcommand{\gb}{\beta}
\newcommand{\gd}{\delta}
\newcommand{\gep}{\varepsilon}       

\newcommand{\gG}{\Gamma}

\newcommand{\go}{\omega}

\newcommand{\bP}{{\ensuremath{\mathbf P}} }
\newcommand{\bE}{{\ensuremath{\mathbf E}} }

\newcommand{\bbbP}{\mathbb{P}}
\newcommand{\E}{\mathbb{E}}

\newcommand{\N}{\mathbb{N}}
\newcommand{\ind}{\mathbf{1}}

\renewcommand{\cc}{\complement}

\DeclareMathSymbol{\leqslant}{\mathalpha}{AMSa}{"36} 
\DeclareMathSymbol{\geqslant}{\mathalpha}{AMSa}{"3E} 
\DeclareMathSymbol{\eset}{\mathalpha}{AMSb}{"3F}     
\renewcommand{\leq}{\;\leqslant\;}                   
\renewcommand{\geq}{\;\geqslant\;}                   
\newcommand{\dd}{\,\text{\rm d}}             
\DeclareMathOperator*{\union}{\bigcup}       
\newcommand{\sumtwo}[2]{\sum_{\substack{#1 \\ #2}}} 

\title[Marginal $\gamma$-stable pinning model]{Marginal relevance for the $\gamma$-stable pinning model.}

\author{Hubert Lacoin}
\address{
  IMPA, Institudo de Matem\'atica Pura e Aplicada, Estrada Dona Castorina 110
Rio de Janeiro, CEP-22460-320, Brasil. 
}

\begin{document}

\begin{abstract}
We investigate disorder relevance for the pinning of a renewal 
when the law of the random environment is in the domain of attraction of a stable law with parameter $\gamma \in (1,2)$.
Assuming that the renewal jumps have power-law decay, we
determine under which condition the critical point of the system modified by the introduction of a small quantity of disorder.
In an earlier study of the problem \cite{cf:LS} we have shown that the answer depends on the value of the tail exponent $\alpha$ associated to the distribution of 
renewal jumps:
when $\alpha>1-\gamma^{-1}$ a small amount of disorder shifts the critical point whereas it does not when $\alpha<1-\gamma^{-1}$.
The present paper is focused on the boundary case $\alpha=1-\gamma^{-1}$. We show that a critical point shifts occurs in this case, and 
obtain an estimate for its intensity.
\\[10pt]
\\[10pt]
\end{abstract}
\maketitle

\section{Introduction}

The renewal pinning model has been developed as a toy model to understand phenomena like wetting in two dimension \cite{cf:Ab} 
and pinning of a polymer to a defect line \cite{cf:FLN}.
Due to its simplicity and the fact that the critical exponent associated to the localization transition can be tuned to any value just by modifying parameter 
(the tail exponent of the renewal process in \eqref{Kt}), 
it has also been employed as benchmark to test prediction concerning the effect of disorder obtained renormalization group arguments. We refer to the monographs 
\cite{cf:GB, cf:GB2} for a complete introduction to the subject.

\medskip

More precisely a rich literature has been developed (see
\cite{cf:A06, cf:AZ08, cf:CdH,  cf:DGLT07, cf:DHV,  cf:GT05, cf:Lmart, cf:T08} and references therein), in order to establish rigorously that 
the sensibility of the system to disorder is determined by the sign of the critical exponent associated to  the specific heat 
the as predicted by Harris \cite{cf:Harris}.
More precisely it was shown that when the specific-heat exponent is positive (which corresponds to $\alpha>1/2$ for the exponent in \eqref{Kt})
disorder even of small intensity shifts the critical point and modifies the critical exponent, while when it is negative ($\alpha<1/2$) the 
critical point and the critical exponent of the localization transition are conserved

\medskip

The criterion developed by Harris does not yield any prediction when the specific heat exponent vanishes:
this corresponds to a tails exponent $\alpha=1/2$ for the renewal process. This case is of special importance in the case of pinning as it corresponds 
to the original random walk pinning model (see e.g.\ \cite{cf:Fisher}). A more detailed renormalization group analysis in \cite{cf:DHV} yielded that in 
this so-called marginal case,
disorder should also be relevant (a prediction conflicting with others made in the literature e.g. \cite{cf:FLNO}, see the introduction of \cite{cf:GLT08} 
for a more detailed account on the controversy). This conjecture was proved in \cite{cf:GLT08} (see also \cite{cf:BL, cf:GLT09}).

\medskip

As most heuristics concerning disorder relevance rely on second moment expansion, a natural question is: 
\begin{center}
Is Harris criterion valid when the disorder has infinite variance? 
\end{center}
The issue was raised for pinning model in \cite{cf:LS} and it was shown that when the disorder is in the domain of attraction of a $\gamma$-stable law 
$\gamma\in (1,2)$,
Harris criterion is not satisfied. More precisely we showed that the critical point is shifted when $\alpha>1-\gamma^{-1}$ and that 
that critical points and exponents are not perturbed by a small amount of disorder when $\alpha<1-\gamma^{-1}$.

 \medskip
 
 In the present work we investigate the marginal case $\alpha=1-\gamma^{-1}$ for which we  prove disorder relevance.
It presents strong analogies with the Random Walk pinning model treated in \cite{cf:DHV} and \cite{cf:GLT08}.
 While the methods used to resolve it are clearly inspired by those used in the marginal case with second moment \cite{ cf:BL, cf:GLT08, cf:GLT09}, 
they also incorporate new ingredients which are necessary to deal with heavier tail disorder.

 \section{Model and results}

   \subsection{Disordered pinning and phase transition}

  Consider $\tau = (\tau_{n})_{n \geq 0}$ a recurrent integer valued renewal process, 
  that is a random sequence starting from $\tau_0=0$ whose increments $(\tau_{n+1}-\tau_{n})$ are independent,
   identically distributed (IID) positive integers. We let $\bP$ denote the associated probability distribution 
and assume that the inter-arrival distribution has power-law decay or more precisely
  \begin{equation}\label{Kt}
   K(n) :=  \bP[\tau_{1} = n] \stackrel{n\to \infty}{\sim} C_K n^{-(1+\ga)}, \quad \ga \in (0,1),
  \end{equation} 
where $C_K>0$ is an arbitrary constant. 
Note that $\tau$ can alternatively be considered as an infinite
   subset of $\bbN$ and in our notation $\{ n\in \tau \}$ is equivalent to $\{ \exists k\in \bbN, \ \tau_k=n\}$.

 \medskip

\noindent We consider a sequence of IID random variables $(\go_{n})_{n\ge 0}$ and denote its law by $\bbP$ which satisfies
 $\E[\go_1]=0$ and for some $a\in(0,1)$
\begin{equation}\label{star}
 \bbP[\go_1 \ge  -a]=1.
\end{equation}
We work under the assumption that $\go$ is in the domain of attraction of a $\gamma$-stable law, or more precisely we assume that
for some $C_{\bbP}>0$ we have
\begin{equation}\label{defgamma}
    \bbP[\go_n\ge x]\stackrel{x\to \infty}{\sim} C_{\bbP}x^{-\gamma}, \quad \gamma\in (1,2).
\end{equation}
Given $\gb  \in [0,1]$, $h \in \bbR$, and $N  \in \N$, we define 
    a modified renewal measure $\bP^{\gb,\go}_{N,h}$ whose 
    Radon-Nikodym derivative with respect to $\bP$ is given by 
   \begin{equation}\label{defmod}
    \frac{\dd\bP^{\gb,\go}_{N,h}}{\dd \bP}(\tau) = \frac 1 {Z^{\gb,\go}_{N,h}}  \left(\prod_{n\in [1,N]\cap \tau}e^{h}(\gb\go_n+1)\right)\ind_{\{N\in \tau\}},
   \end{equation} 
   where
    \begin{equation}\label{partfunc}
     Z^{\gb,\go}_{N,h} = \bE\left[  \left(\prod_{n\in [1,N]\cap \tau}e^{h}(\gb\go_n+1)\right)\ind_{\{N\in \tau\}} \right].
    \end{equation} 
 In the case $\gb=0$, we retrieve the homogeneous pinning model which, setting $\delta_n:= \ind_{\{n\in \tau\}}$, is defined by 
    \begin{equation}
        \frac{\dd\bP_{N,h}}{\dd \bP}(\tau):=\frac{1}{Z_{N,h}}e^{h\sum_{n=1}^N\gd_n}\gd_N
        \quad \text{and} \quad Z_{N,h}:=\bE\left[e^{h\sum_{n=1}^N \gd_n}\gd_N\right].
    \end{equation} 
    We investigate the behavior of $\tau$ under $\bP^{\gb,\go}_{N,h}$ using the notion of \textit{free energy per monomer}, 
    which is defined as the asymptotic growth rate of the partition function
    \begin{equation}
    \tf(\gb,h) := \lim_{N \to \infty} \frac{1}{N} 
    \log Z^{\gb,\go}_{N,h} \stackrel{\bbbP- a.s.}{=} \lim_{N \to \infty} \tfrac{1}{N} 
    \E\left[ \log Z^{\gb,\go}_{N,h} \right]<\infty.
    \end{equation}  
We refer to \cite[Theorem 4.1]{cf:GB} for a proof of existence of   $\tf(\gb,h)$. 
Note that $\tf(\gb,h)$ is non-negative, and that $h\mapsto \tf(\gb,h)$ is non-decreasing and convex (as a limit of non decreasing convex functions).
By exchanging  limit and  derivative, as allowed by convexity, we obtain that the derivative of $\tf$ w.r.t. $h$  corresponds to the asymptotic contact fraction
   \begin{equation}\label{contacts}
    \partial_h \tf(\gb,h):= \lim_{N\to \infty} \frac{1}{N}\bE^{\gb,\go}_{N,h}\left[ | \tau\cap[1,N]|\right].
   \end{equation}
In particular, if one sets 
 \begin{equation}
   h_c(\gb):= \inf\{ h\in \bbR \ | \  \tf(\gb,h)>0 \},
  \end{equation}
we have 
\begin{equation}\begin{split}
  \limsup_{N\to \infty} \frac{1}{N}\bE^{\gb,\go}_{N,h}\left[|\tau\cap[1,N]|\right]=0 \quad \text{ if } h<h_c(\gb),\\
  \liminf_{N\to \infty} \frac{1}{N}\bE^{\gb,\go}_{N,h}\left[\left| \tau\cap[1,N]\right|\right]>0 \quad \text{ if } h>h_c(\gb).
\end{split}\end{equation}
We say in the first case that $\tau$ is delocalized and in the second one that it is localized.
It can be proved using simple inequalities (see below or \cite[Proposition 5.1]{cf:GB}), that $h_c(\gb)\notin\{-\infty,\infty\}$ meaning 
that this phase transition really occurs.

\subsection{Annealed comparison and disorder relevance}

Using Jensen's inequality and the assumption that the $\go$s have zero mean we have 
\begin{equation}\label{annehilde}
 \bbE \left[ \log Z^{\gb,\go}_{N,h} \right]\le \log  \bbE \left[  Z^{\gb,\go}_{N,h} \right]=\log Z_{N,h}.
\end{equation}
Hence
\begin{equation}\label{ancompa}
 \forall \gb\in(0,1], \quad \tf(\gb,h)\le \tf(h),
\end{equation}
Our assumption \eqref{star} also implies that $\tf(\gb,h)\ge \tf(h+\log(1-a \gb ))$.

\medskip

The localization transition is easier to analyze when $\gb=0$, and this makes the inequality \eqref{ancompa} more interesting:
$\tf(h)$ is the solution of an explicit inverse problem
    \begin{equation}
     \tf(h)=\begin{cases}
             0 \quad &\text{ if } h\le 0,\\
             g^{-1}(h) \quad &\text{ if } h>0,
             \end{cases}
           \end{equation}
where $g$ is defined on $\bbR_+$ by
$$g(x):=-\log \left( \sum_{n=1}^{\infty} e^{-nx} K(n) \right).$$
In particular we have $h_c(0)=0$ and from a closer analysis of $g$ (see \cite[Theorem 2.1]{cf:GB})
$$\tf(h)\stackrel{h\to 0+}{\sim} \left(\frac{\alpha h}{C_K\gG(1-\alpha)}\right)^{\frac{1}{\alpha}}.$$
A natural question is to ask whether the annealed comparison \eqref{ancompa} is sharp, in the following sense:
\begin{itemize}
 \item [(A)] Is the critical preserved when disorder is introduced:\\
  Do we have $h_c(\gb)=0$ \ ?
 \item [(B)] Is the critical exponent for the phase transition preserved:\\
 Do we have $\tf(h,\gb)\approx h^{1/\alpha}$ in some sense \ ?
\end{itemize}
If these two property hold, it means that the introduction of disorder in the system does not change its property and this situation is
referred to as \textsl{irrelevant disorder}. In the case where the critical properties of the system are changed disorder is said to be relevant.

\subsection{Harris criterion and former results}

A.B Harris \cite{cf:Harris} developed a criterion in order to predict disorder.
For one dimensional systems such as the one studied in the present paper, 
it can be interpreted as follows if the critical exponent for free-energy of the pure (i.e.\ $\gb=0$) model is larger than 
$2$ then disorder is irrelevant for small $\gb$ whereas disorder is always relevant in the case when the exponent is larger than $2$.
In the case of pinning model, this means that disorder is irrelevant for $\alpha<1/2$ and relevant for $\alpha>1/2$.

\medskip

The validity of the Harris criterion has been confirmed in various cases for the pinning model, in the case where 
the environment has finite second moment $\bbE[\go^2_1]<\infty$ (see \cite{cf:A06, cf:Lmart, cf:T08} for the irrelevant disorder case, and 
\cite{cf:AZ08, cf:DGLT07, cf:GT05} in the relevant case). This assumption is far from being only technical as Harris heuristics is based on 
a second moment expansion at the vicinity of the critical point in order to test stability.

\medskip

For this reason we suspected that with an environment with an  heavier tail distribution, Harris criterion may not be valid.
This has been confirmed in \cite{cf:LS} where we have shown that disorder is irrelevant when  $\alpha<1-\gamma^{-1}$ and relevant for $\alpha>1-\gamma^{-1}.$

\begin{theorema}[From Theorems 2.3 and 2.4 in \cite{cf:LS}] \ \\

\begin{itemize}
\item[(A)] If $\alpha<1-\gamma^{-1}$, then there exists  $\gb_0$ such that for all $\gb\in (0,\gb_0]$  
we have \\ $h_c(\gb)=0$ and furthermore
\begin{equation}
 \lim_{h\to 0+} \frac{\log \tf(\gb,h) }{\tf(h)}=\frac{1}{\alpha}.
\end{equation}
\item[(B)] If $\alpha>1-\gamma^{-1}$, then for all $\gb$ we have $h_c(\gb)>0$ and 
\begin{equation}\label{exponenta}
 \lim_{\gb\to 0+} \frac{\log h_c(\gb) }{\log \gb}=\frac{\alpha\gamma}{1-\gamma(1-\alpha)}.
\end{equation}
\end{itemize}
\end{theorema}

These results indicate that Harris criterion has to be reinterpreted in the case where the environment is heavy-tailed.
A question which has been left open in  \cite{cf:LS} is the case $\alpha=1-\gamma^{-1}$ which we refer to as the \textsl{marginal case}.

\subsection{Main result}

The main achievement of this paper is to prove that disorder shifts the critical point for all values of $\gb$ also in the marginal case 
 $\alpha=1-\gamma^{-1}$. The result bears some similarity with the one proved in \cite{cf:GLT08}, when it is shown that under finite second moment assumption for $\go$
 (\cite{cf:GLT08} actually only treats the case of Gaussian environment but the generalization can be found in \cite{cf:GLT09}), disorder is relevant 
 when the renewal exponent satisfies $\alpha=1/2$.

 \begin{theorem}
 Assume that \eqref{Kt} and \eqref{defgamma} are satisfied for $\alpha=1-\gamma^{-1}$.
Then, for any $\gb\in[0,1]$, $h_c(\gb)>0$ and furthermore,  there exists a constant $A>0$ such that
 \begin{equation}\label{lowermargi}
\forall \gb\in(0,1], \quad  h_c(\gb)\ge \exp(-A \gb^{-2\gamma}).
 \end{equation}
\end{theorem}

  \begin{rem}
 We are discussing in this paper only the case where the inter-arrival distribution $K(\cdot)$ has a pure power-law behavior, cf.\ \eqref{Kt}.
 When a slowly varying function is introduced instead of the constant $C_K$, the picture gets slightly more complicated and a necessary and 
 sufficient condition for disorder relevance 
 was proved in \cite{cf:BL} under the finite second moment assumption. For $\gamma$-stable environment we refer to \cite[Section 2.5.1]{cf:LS} for a conjecture.
 \end{rem}
 
\begin{rem}
 We do not believe that this lower bound on $h_c(\gb)$ is optimal but we know from \cite[Proposition 6.1]{cf:LS} that $h_c(\gb)$ is smaller than any power of $\gb$ 
 at the vicinity of $0$. This contrasts with the case $\alpha>1-\gamma^{-1}$, cf.\ \eqref{exponenta}.
 It seems plausible that improving the technique presented in the present paper in the same spirit as what is done in \cite{cf:BL}, 
 we can bring the exponent in the exponential in \eqref{lowermargi} from $2\gamma$ down to $\gamma$, which could be the optimal answer.
 We would not know however how to obtain a matching upper bound.
 \end{rem}
 
\subsection{Organization of the paper}

The proof of our main statement is divided into three main steps: in Section \ref{corse} 
we present a sequence of inequalities which combines coarse graining ideas (in a very similar spirit with what has been done e.g.\ 
in \cite{cf:BL,cf:GLT08}) and a change of measure which penalizes environment which displays atypical ``dual peaks''. This reduces the problem to
to estimating the coarse-grained partition function under the penalized measure (Proposition \ref{prop:bouzouf}), provided we control the ``cost'' 
of the penalization procedure (Proposition \ref{prop:kist}).
In Section \ref{kb}, Proposition \ref{prop:kist} is proved while Proposition \ref{prop:bouzouf} is reduced to a one block estimate (Proposition \ref{prop:binif}), which is itself proved in 
Section \ref{oneB}

 \section{Fractional moments, coarse graining and change of measure} \label{corse}

 \subsection{Fractional moments}
 
Let us consider $\theta\in (0,1)$. A more efficient bound than \eqref{annehilde} can be achieved on the free-energy by applying Jensen's inequality in a different manner.
 \begin{equation}\label{annehil2}
 \bbE \left[ \log Z^{\gb,\go}_{N,h} \right]= 
 \frac{1}{\theta}  \bbE \left[  \log \left(Z^{\gb,\go}_{N,h}\right)^{\theta} \right] \le  
 \frac{1}{\theta}\log  \bbE \left[  \left(Z^{\gb,\go}_{N,h}\right)^{\theta} \right].
\end{equation}
In particular we can prove that $\tf(\gb,h)=0$ if we have 
\begin{equation}\label{daliminf}
 \liminf_{N\to \infty}\frac{1}{N}\log  \bbE \left[  \left(Z^{\gb,\go}_{N,h}\right)^{\theta} \right]=0.
\end{equation}
We set 
\begin{equation}\label{mdef}
h_{\gb}:=\exp\left(- A \gb^{\gamma^{-1}}\right),
\end{equation}
where $A>0$ is a sufficiently large constant,
and consider a special length
\begin{equation}\label{skoop}
 \ell_{\gb}:= h^{-1}.
\end{equation}
We consider a system whose size $N=m \ell$ is an integer multiple of $\ell$.

\begin{proposition}\label{damainpropz}
Given $\theta\in \left( \frac{\gamma}{2\gamma-1},1\right)$, if $A$ is chosen sufficiently large, we have for all $\gb\in (0,1]$
 \begin{equation}\label{eq:frakick}
  \limsup_{m\to \infty} \bbE\left[ \left(Z^{\gb,\go}_{m\ell_{\gb} ,h_{\gb}}\right)^\theta\right]<\infty.
 \end{equation}
In particular we have 
$$h_c(\gb)\ge h_{\gb}.$$
 \end{proposition}
 
The proof of the proposition goes in two steps: Firstly, we use a kind of bootstrapping argument in order to reduce the problem to estimates of
partition functions of systems of size smaller or equal to $\ell$.
Secondly, to control the partition function of these smaller system we introduce a change of measure procedure which has the effect of penalizing some atypical 
environments whose contribution to the annealed partition function is significant.

\medskip

The approach adopted in \cite{cf:LS} to prove disorder relevance when $\alpha> 1-\gamma^{-1}$
used a finite volume criterion from \cite{cf:DGLT07}, and penalized environments for which too large values of $\go$ appeared.
This approach fails to give any result in the present case
and we need to perform a finer analysis to catch the critical point shift.

\medskip

We introduce two improvement with respect to the method used in \cite{cf:LS}: the first is to replace \cite[Proposition 2.5]{cf:DGLT07} by a finer coarse graining.
This is not a new idea and is very similar to the method applied e.g.\ in \cite{cf:T09}.
The second improvement is the main novelty of this paper and concerns the type of penalization considered in the change of measure procedure: 
we design a new form of penalization which involves considering pairs of site where $\go$ displays high values.

\medskip

This approach contrasts with what has been done in the marginal case under finite second moment assumption:
in the case of Gaussian environment,  a penalization that would induce a change of the covariance structure was considered \cite{cf:GLT08},
and more generally for an environment with finite second moment 
a tilting by a quadratic form, or a multi-linear form of higher order \cite{cf:GLT09,cf:BL} was used in order to prove marginal disorder relevance.
 Under assumption \eqref{defgamma} quadratic forms in $\go$ seems trickier to analyze and we need to select another function of $\go$ which is easier to manipulate.
We decide to look only for extremal values in $\go$ and to penalize environments which 
present two high-peaks close to each other. 
The exact threshold that we use is determined by a function of the distance between the two sites.

\subsection{The coarse graining procedure}

For the sake of completeness let us repeat in full details the coarse graining procedure from \cite{cf:BL}. 
We split the system into blocks of size $\ell$, we define for $i\in \lint 1, m\rint$
\begin{equation}
 \ B_i:=\lint \ell (i-1)+1, \ell i \rint.
\end{equation}
Given $\cI = \{i_1,\ldots, i_{|\cI|}\} \subset \lint 1, m\rint$
we define the event 
\begin{equation}
E_{\cI}:= \Big\{ \, \big\{ i\in \lint 1, m\rint  \ : \ \tau\cap B_i\ne \emptyset \, \big\}= \mathcal I  \, \Big\}\, ,
\end{equation}
and set $Z^{\cI}$ to be the contribution to the partition function of the event $E_{\cI}$,
\begin{equation}
Z^{\cI}:= Z^{\gb,\go}_{N,h}(E_{\cI})=Z^{\gb,\go}_{N,h}\bE^{\gb,\go}_{N,h}[E_{\cI}].
\end{equation}
Note that $Z^{\cI}>0$ if and only if $m\in \cI$. When $\tau\in E_{\cI}$, the set $\cI$ is called the coarse-grained trajectory of $\tau$.
As the $E_{\cI}$ are mutually disjoint events, $Z^{\gb,\go}_{N,h}=\sum_{\cI\subset\lint 1, m\rint}Z^{\cI}$ and thus
using the inequality $(\sum a_i)^{\theta} \leq \sum a_i^{\theta}$ for non-negative $a_i$'s, we obtain
\begin{equation}\label{eq:datsun}
\bbE\left[\left( Z^{\gb,\go}_{N,h}\right)^{\theta}\right]\le \sum_{\cI\subset\{1,\dots, m\}} \bbE\left[\left( Z^{\cI}\right)^{\theta}\right].
\end{equation}
We therefore reduced the proof to that of an upper bound on $\bbE\left[\left( Z^{\cI}\right)^{q}\right]$, 
which can be interpreted as the contribution of the coarse grained trajectory $\cI$ to the fractional moment of the partition function.

\begin{proposition}\label{coarsegrained}
Given $\eta>0$, and $\theta\in (0,1)$, if $A$ is sufficiently large then there exists a constant $C_{\ell}$ such that
for all $\gb \in (0,1]$
for all $m \ge 1$ and $\cI\subset \lint 1,m\rint$,
\begin{equation}\label{smallpining}
  \bbE\left[\left( Z^{\cI}\right)^{\theta}\right]\le C_{\ell} \prod_{k=1}^{|\cI|} \frac{\eta}{(i_{k}-i_{k-1})^{(1+\alpha)\theta}},
\end{equation}
where by convention we have set $i_0:=0$.
\end{proposition}

\begin{proof}[Proof of Proposition \ref{damainpropz} from Proposition \ref{coarsegrained}]
Note that by Jensen inequality we just have to prove the statement for $\theta$ close to one. We choose $\theta<1$ in a way that
$$(1+\alpha)\theta>1.$$
Using \eqref{eq:datsun} and Proposition \ref{coarsegrained} for all $m>0$ we have 
\begin{equation}
\bbE\left[\left( Z^{\gb,\go}_{N,h}\right)^{\theta}\right]\le C_{\ell} 
\sumtwo{\cI\subset  \lint 1, m \rint }{m\in \cI} \prod_{k=1}^{|\cI|} \frac{\eta}{(i_{k}-i_{k-1})^{(1+\alpha)\theta}}.
\end{equation}
By considering the sum over all finite subsets of $\bbN$ with cardinal at most $m$ instead of subsets of $\lint 1, m \rint$ and reorganizing the sum we obtain that 
\begin{equation}\label{zoos}
\bbE\left[\left( Z^{\gb,\go}_{N,h}\right)^{\theta}\right]\le  C_{\ell}  \sum_{j=1}^{m} \left(\eta \sum_{n\ge 1} n^{-(1+\alpha)\theta}\right)^j 
\end{equation}
We can check that choosing 
\begin{equation}
 \eta=\left(2\sum_{n=1}^\infty n^{-(1+\alpha)\theta}\right)^{-1}\,,
\end{equation}
 the l.h.s.\ of \eqref{zoos} is smaller $C_{\ell}$.
 \end{proof}

 \subsection{Penalization of the favorable environments}
 
Let us now introduce the notion of penalization of the environment in a cell,
which is the main tool to prove Proposition \ref{coarsegrained}.

\medskip

\noindent Given $G_{\cI}(\go)$ be a positive function of $(\go_n)_{n\in \union_{i\in \cI} B_i}$, 
using H\"older's inequality, we have 
\begin{equation}
\label{eq:holder}
\bbE\left[ \left(Z^{\cI}\right)^{\theta}\right] \leq 
\left(\bbE\left[G_{\cI}(\go)^{-\frac{\theta}{1-\theta}} \right]\right)^{1-\theta} \ \left(\bbE\left[G_{\cI}(\go) Z^{\cI} \right]\right)^{\theta}.
\end{equation}
We decide to apply this inequality with $G_{\cI}$, which is a product of functions of $(\go_n)_{n\in B_i}$, for $i\in \cI$. More precisely, given $g: \bbR^{\ell} \to \bbR^d$ we set
\begin{equation}\label{eq:strucGI}
 G_{\cI}(\go):=\prod_{i\in \cI} g(\go_{i(\ell-1)+1},\dots,\go_{i\ell})=:\prod_{i\in \cI} g_i(\go).
\end{equation}
We decide to use \eqref{eq:holder} for some $g$ that takes values in $[0,1]$, which should be equal to $1$ for ``typical environments'' but 
close to zero for environments that gives too much contribution to $Z^{\cI}$.
The difficulty lies in finding a function $g$ such that the cost for introducing the penalization 
$\bbE\left[G_{\cI}(\go)^{-\frac{\theta}{1-\theta}} \right]$ is not too big, and such that 
$\bbE\left[G_{\cI}(\go) Z^{\cI} \right]$ is much smaller than $\bbE\left[Z^{\cI} \right]$ (so that we get a large benefit out of it).

\medskip

Let us now introduce our choice for the function $g$.  Instead of giving a fixed penalty (i.e multiplication by some factor smaller than $1$)
for each $\go_n$ above a certain threshold (something of order $\ell^{\gamma^{-1}}$) like in \cite{cf:LS}, which would not give any conclusive 
result in the case presently studied,
we decide to introduce a $g$ that penalizes the presence of ``dual peaks in the environment'' $(\go_n)_{n\in \lint 1,\ell \rint}$.
Given $M$ a large constants, we set 
\begin{equation}
 g(\go_1,\dots,\go_{\ell}):=\exp\left(-M \ind_{\cA_{\ell} } \right),
\end{equation}
where 
\begin{equation}
 \cA_{\ell}:=\big\{ \exists i ,j \in \lint 1, \ell \rint \ : \  \min(\go_i ,\go_j)\ge  V(M,\ell,i-j)  \big\}, 
\end{equation}
and
\begin{equation}\label{defV}
V(M,\ell,n)=V(n):= e^{M^2}(\ell (\log \ell) n)^{\frac{1}{2\gamma}}. 
\end{equation}
We are going to prove that with this choice for $g$, the benefits of the penalization overcomes the cost. This is the object of the two following results, whose proofs are postponed 
to the next section. 

\begin{proposition}\label{prop:kist}
Given $\theta$, and $M>0$ sufficiently large we have
\begin{equation}\label{eq:glocost}
\bbE\left[G_{\cI}(\go)^{-\frac{\theta}{1-\theta}} \right]\le 2^{|\cI|}.
\end{equation}
\end{proposition}

\begin{proposition}\label{prop:bouzouf}
Given $\eta>0$, there exists $M$ and $A$ such that for all $\gb\in(0,1]$ we have 

\begin{equation}\label{kouz}
\bbE\left[G_{\cI}(\go) Z^{\cI} \right]\le  \frac{C_{\ell} \eta^{|\cI|}}{\prod_{j=1}^{|\cI|} |i_{j}-i_{j-1}|^{1+\alpha}}.
\end{equation}

\end{proposition}
It is quite straightforward using \eqref{eq:holder} to check that Proposition \ref{coarsegrained} 
is a consequence of Proposition \ref{prop:kist} and \ref{prop:bouzouf} with adequate changes for the value of $\eta$ and $C_{\ell}$.

\section{The costs and benefits of the penalization procedure}\label{kb}

\subsection{The proof of Proposition \ref{prop:kist}}

Using the fact that the environment is IID and the block structure of $G_{\cI}$, it is sufficient to show that 
\begin{equation}
 \bbE\left[ g(\go_1,\dots,\go_{\ell})^{\frac{\theta}{1-\theta}} \right]\le 2.
\end{equation}
We have 
\begin{equation}\label{woop}
 \bbE\left[ g(\go_1,\dots,\go_{\ell})^{\frac{\theta}{1-\theta}} \right]\le 1+  e^{\frac{M\theta}{1-\theta}}\bbP[\cA_{\ell}].
\end{equation}
Using a union bound and the tails distribution of the $\go$ \eqref{defgamma} (recall also \eqref{defV}) 
the probability above can be bounded as follows

\begin{multline}
 \bbP\left[  \cA_{\ell} \right]
\le   \sum_{1\le i<j \le \ell} \bbP\left[ \min(\go_i ,\go_j)\ge  V(j-i) \right]\\
\le  C  \frac{e^{-2\gamma M^2}}{\ell (\log \ell)} \sum_{1\le i<j \le \ell} \frac{1}{(j-i)}\le C'e^{-2\gamma M^2}
\end{multline}
Hence if $M$ is sufficiently large, the second term in \eqref{woop} is sufficiently small and we can conclude.
\qed

\subsection{The proof of Proposition \ref{prop:bouzouf}}

The main tool to prove Proposition \ref{prop:bouzouf} is the following results that shows that in each block where it is performed,
the penalization is effective.

\medskip

\noindent For any couple of integers $a<b$ we define
\begin{equation}
Z^h_{[a,b]}:= \bE\left[ \prod_{n\in \tau \cap [a, b] } e^h(1+\gb \go_n)  \ \middle| \ a,b\in\tau \right].
\end{equation}
We have 
\begin{equation}
Z^h_{[a,b]}= \frac{\bE\left[ \prod_{n\in \tau \cap [0, b-a] } e^h(1+\gb \go_{a+n}) \right]}{u(b-a)},
\end{equation}
where $u(n):=\bP[n\in \tau]$. Let us mention an asymptotic equivalent of $u(n)$ \cite[Theorem 1]{cf:doney},
which holds under assumption \ref{Kt} and is used in the rest of the proof
\begin{equation}\label{doney}
u(n)\stackrel{n\to \infty}{\sim} \frac{\alpha \sin(\pi \alpha)}{\pi C_K} (n+1)^{\alpha-1}.
\end{equation}
The main tool to prove Proposition \ref{prop:bouzouf} is the following result which quantifies how the multiplication by 
$g$ affects the expected value of the partition functions in 
a single block.

\begin{proposition}\label{prop:binif}
 Given $\eta$ and $M$, if $A$ is sufficiently large then for any $\gb\in(0,1]$, and $(d  ,  f) \in \lint 1, \ell \rint^2$ satisfying $(f-d)\ge \eta \ell$, we have 
 \begin{equation}\label{binn}
 \bbE\left[ g(\go_1,\dots,\go_{\ell})Z^0_{[d,f]}\right]\le \eta.
 \end{equation}
 \end{proposition}

 We prove this result in the next section and show here simply how \eqref{kouz} can be deduced from it.
 We decompose $Z^{\cI}$ according to the first and last contact points in each block $(B_{i})_{i\in \cI}$ 
 where $\cI:=\{i_1,\dots, i_l\}$. We have
\begin{multline}
\label{def:hatZ}
Z^{\cI}:= \sumtwo{d_1,f_1 \in B_{i_1}}{d_1\leq f_1}  \cdots \sumtwo{d_l \in B_{i_l}}{f_l=N}
K(d_1) u(f_1-d_1) Z^h_{d_1,f_1} K(d_2-f_1) \\
 \cdots\  K(d_l-f_{l-1}) u(N-d_l ) Z^h_{d_l,N} \, .
\end{multline}
Then we use the fact that that due to our choice for the value of $h$, we have  
$Z^h_{d,f}\le e^{\ell h}Z^0_{d,f}=e Z^0_{d,f}$, for any $d$ and $f$ such that $(f-d)\le \ell$.
We obtain thus using the product structure of $G_{\cI}$ \eqref{eq:strucGI}
 
 \begin{multline}
\label{def:hatZ2}
\bbE\left[G_{\cI}(\go)Z^{\cI}\right]\le  e^{|\cI|}
\sumtwo{d_1,f_1 \in B_{i_1}}{d_1\leq f_1}  \cdots \sumtwo{d_l \in B_{i_l}}{f_l=N}
K(d_1) u(f_1-d_1) \bbE\left[ g_{i_1}(\go)Z^h_{[d_1,f_1]}\right] K(d_2-f_1) \\
 \cdots\  K(d_l-f_{l-1}) u(N-d_l ) \bbE \left[ g_{i_l}(\go) Z^h_{[d_l,N]}\right]
 \\
 \le  e^{|\cI|} \sumtwo{d_1,f_1 \in B_{i_1}}{d_1\leq f_1}  \cdots \sumtwo{d_l \in B_{i_l}}{f_l=N}
K(d_1) u(f_1-d_1) \left[\eta+(1-\eta)\ind_{\{(f_1-d_1)\le \eta \ell\} }\right] K(d_2-f_1) \\
 \cdots\  K(d_l-f_{l-1}) u(N-d_l )\left[\eta+(1-\eta)\ind_{\{(f_l-d_l)\le \eta \ell\}}\right],
\end{multline}
 where in the last line we used Proposition \ref{prop:binif}, and when $(f_j-d_j)\le \eta \ell$, the fact that 
  \begin{equation}
 \bbE\left[ g_{i_j}(\go)Z_{[d_j,f_j]}\right]\le  \bbE\left[ Z_{[d_j,f_j]}\right]=1.
 \end{equation}
Now we only need to obtain a bound on the r.h.s\ of \eqref{def:hatZ2}.
Using \eqref{Kt} and \eqref{doney}, we can replace $K(n)$ and $u(n)$ by $n^{-(\alpha+1)}$ and $(n+1)^{1-\alpha}$ 
($n+1$ is present instead of $n$ because we also consider 
$u(0)$) at the cost of loosing a constant factor per cell.
Thus we only need to prove that given $\delta>0$, if $\eta$ is sufficiently small, we have for some constant $C_{\ell}$ 
\begin{multline}\label{eq:gross}
 \sumtwo{d_1,f_1 \in B_{i_1}}{d_1\leq f_1}  \cdots \sumtwo{d_l \in B_{i_l}}{f_l=N}
(d_1)^{-(1+\alpha)} (f_1-d_1+1)^{\alpha-1} [\eta+(1-\eta)\ind_{\{(f_1-d_1)\le \eta \ell\} }] (d_2-f_1)^{-(1+\alpha)} \\
 \cdots\  (d_l-f_{l-1})^{-(1+\alpha)} (N-d_l+1)^{\alpha-1} [\eta+(1-\eta)\ind_{\{(f_l-d_l)\le \eta \ell\}}]
\\ \le C_{\ell} \delta^{|\cI|}\prod_{j=1}^l |i_{j}-i_{j-1}|^{-(1+\alpha)}.
\end{multline}
To obtain the bound, we proceed as
in the computation \cite[Equations (4.25) to (4.37)]{cf:BL}.
We are to prove that for all $i\in \{1,\dots,l-1\}$ we have, for any $f_{i-1}\in B_{i-1}$ (with the convention $f_0=0$) and $d_{i+1}\in B_{i+1}$,
provided that $\eta$ is sufficiently small
\begin{multline}\label{themidek}
  \sumtwo{d_j,f_j \in B_{i_j}}{d_j\le f_j}  (d_{j}-f_{j-1})^{-\frac{(1+\alpha)}{2}}  (f_{j}-d_{j}+1)^{\alpha-1} 
  [\eta+(1-\eta)\ind_{\{(f_i-d_i)\le \eta \ell\} }] (d_{j+1}-f_{j})^{-\frac{(1+\alpha)}{2}}
  \\
  \le \delta\left[ (i_j-i_{j-1})(i_{j+1}-i_j)\right]^{-(1+\alpha)}.
\end{multline}
Additionally we need two additional inequalities to bound the contribution of the first and last jump respectively.
The reader will readily check that
\begin{equation}\label{thene}
\begin{split}
  (d_{i_1})^{-\frac{(1+\alpha)}{2}}&\le i_1^{-\frac{(1+\alpha)}{2}},\\
 \sum_{d_l \in B_{m}} (d_l-f_{l-1})^{-\frac{(1+\alpha)}{2}} (N-d_l+1)^{\alpha-1}&\le \ell (m-i_{l-1})^{-\frac{(1+\alpha)}{2}}.
\end{split}
\end{equation}
Equation \eqref{eq:gross} follows by multiplying the three inequalities given in \eqref{themidek} and \eqref{thene}.

\medskip

Let us now prove \eqref{themidek}:
we split the set of indices $\{d_j,f_j \in B_{i_j} \ : \ d_j\le f_j\}$ in the r.h.s.\ of \eqref{themidek} into three subsets by adding an extra condition:
\begin{itemize}
 \item [(i)] $\{d_j,f_j \in B_{i_j} \ : \  (i_j-3/4)\ell \le d_j\le f_j\}$,
 \item [(ii)] $\{d_j,f_j \in B_{i_j} \ : \  d_j\le f_j\le  (i_j-3/4)\}$,
 \item [(iii)] $\{d_j,f_j \in B_{i_j} \ : \  f_j\ge d_j+\ell/2 \}$.
\end{itemize}
It is easy to check that the union of these (non disjoint) subsets give us back the original set.
We estimate the contribution of each set separately, the idea being that each condition in $(i)-(iii)$ allows 
to replace one of the variable factors by an asymptotic equivalent which does not depend on $d_i$ nor $f_i$.   This makes the computation easier.
First we can bound the contribution $(i)$ as follows
\begin{multline}
 \sumtwo{d_j,f_j \in B_{i_j}}{(i_j-3/4)\ell \le d_j\le f_j} \dots\\
 \le 
 [\ell(i_j-i_{j-1})/4]^{-\frac{(1+\alpha)}{2}} \sumtwo{d_j,f_j \in B_{i_j}}{d_j\le f_j} (f_{j}-d_{j}+1)^{\alpha-1}  
 [\eta+(1-\eta)\ind_{\{(f_1-d_1)\le \eta \ell\} }] (d_{j+1}-f_{j})^{-\frac{(1+\alpha)}{2}}
\end{multline}
Then considering the sum over $d_j$ separately, we obtain that the remaining double sum is smaller than 
\begin{equation}
 \left(\sum_{a=1}^{\eta\ell}a^{\alpha-1} + \eta \sum_{a=\eta \ell}^{\ell} a^{\alpha-1} \right)
 \left(\sum_{f_j\in B_{i_j}} (d_{j+1}-f_{j})^{-\frac{(1+\alpha)}{2}}\right).
\end{equation}
The first factor is smaller than $\gep \ell^{\alpha}$ where $\gep$ can be made arbitrarily small by considering small $\eta$, and 
the second factor is of order $|i_{j+1}-i_{j}|^{-\frac{(1+\alpha)}{2}} \ell^{\frac{1-\alpha}{2}}$.
All the powers of $\ell$ cancel out and we obtain that 
\begin{equation}
 \sumtwo{d_j,f_j \in B_{i_j}}{(i_j-3/4)\ell \le d_j\le f_j} \dots \le \frac{\delta}{3} \left[ (i_j-i_{j-1})(i_{j+1}-i_j)\right]^{-\frac{(1+\alpha)}{2}},
 \end{equation}
where the term in the sum in the same as in \eqref{themidek}.
We obtain similarly by symmetry
\begin{equation}
 \sumtwo{d_j,f_j \in B_{i_j}}{ d_j\le f_j\le  (i_j-1/4)\ell} \dots \le \frac{\gamma}{3} \left[ (i_j-i_{j-1})(i_{j+1}-i_j)\right]^{-\frac{(1+\alpha)}{2}}.
 \end{equation}
 Finally in the case $f_j-d_j\ge \ell/2$ we have  $$[\eta+(1-\eta)\ind_{\{(f_j-d_j)\le \eta \ell\} }] (f_{j}-d_{j})^{\alpha-1}\le \eta(\ell/2+1)^{\alpha-1}$$
and thus
 \begin{equation}
 \sumtwo{d_j,f_j \in B_{i_j}}{f_i\ge d_j +\ell/2} \dots 
 \le  \delta(\ell/2+1)^{1-\alpha}\sum_{d_j, f_j \in B_{i_j}} (d_{j}-f_{j-1})^{-\frac{(1+\alpha)}{2}}(d_{j+1}-f_{j})^{-\frac{(1+\alpha)}{2}}.
\end{equation}
The double sum factorizes and can be shown to be of order 
$$\ell^{1-\alpha}   \left[ (i_j-i_{j-1})(i_{j+1}-i_j)\right]^{-\frac{(1+\alpha)}{2}}.$$
This yields (when $\eta$ is sufficiently small)
\begin{equation}
 \sumtwo{d_j,f_j \in B_{i_j}}{f_j\ge d_j +\ell/2} \dots \le   \frac{\delta}{3} \left[ (i_j-i_{j-1})(i_{j+1}-i_j)\right]^{-\frac{(1+\alpha)}{2}}.
\end{equation}
This concludes the proof of \eqref{themidek} and thus of Proposition \ref{prop:bouzouf}.

\section{Proof of Proposition \ref{prop:binif}}\label{oneB}

\subsection{Reduction to a simpler statement}

The aim of this section is to reduce the proof of Proposition \ref{prop:binif} to the estimation of the probability of some nice event for the environment $\go$.
As $\bbE\left[ Z^0_{[d,f]}\right]=1$, $Z^0_{[d,f]}$
can be considered as a probability density.
To prove \eqref{binn} we must thus show that the probability of $\cA_{\ell}$ 
under the probability $Z^0_{[d,f]}(\go) \bbP[\dd \go]$ is close to one, whenever $(f-d)\ge \eta \ell$.

\medskip

\noindent More precisely, given a fixed realization of $\tau$, with $a,b\in \tau$ we define $\bbP^{a,b}_{\tau}$
\begin{equation}
 \frac{\dd \bbP^{d,f}_{\tau}}{\dd \bbP}(\go):= \prod_{n\in \tau \cap \lint d, f\rint} (1+\gb\go_n).
\end{equation}
We have 
\begin{equation}
   \bbE \left[  g(\go_1,\dots,\go_{\ell}) Z^h_{d,f} \right]=
   e^{-M}+ (1-e^{-M}) \bE\left[   \bbP^{d,f}_{\tau}(\cA^{\cc}_{\ell}) \ | \ d,f\in \tau \right].
\end{equation}
We notice that under $\bbP^{d,f}_{\tau}$, the $\go_n$s are still independent, but they are not identically distributed anymore, 
as for  $n\in [d,f]\cap \tau$, the distribution of $\go_n$ has been tilted and thus peaks are more likely to appear on those sites. 
In order to bound the probability of $\cA_{\ell}$
we are going to check only sites with tilted environment.
Let us consider the alternative event (recall \eqref{defV})
\begin{equation}
\cA (d,f,\tau):= \left\{ \exists i,j \in \tau \cap [d,f] , \ \min(\go_i,\go_j) \ge V(j-i) \right\}
\end{equation}
As $\cA (d,f,\tau)$ is clearly included in $\cA_{\ell}$,
it is sufficient for us to obtain a bound on $\bbP^{d,f}_{\tau}(\cA (d,f,\tau)^\cc).$
If we let $\tilde \bbP$ denote the probability obtained by tilting all the variables: the $\go_n$s are IID and with distribution 
\begin{equation}\label{fracs}
 \tilde \bbP[ \go_1\in \dd x]=(1+\gb x) \bbP[ \go_1\in \dd x],
\end{equation}
then we  have 
$$\bbP^{d,f}_{\tau}(\cA (d,f,\tau))=\tilde \bbP\left(\cA (d,f,\tau)\right).$$
Hence can prove Proposition \ref{prop:binif} provided we show
\begin{equation}\label{lapardeA}
  \bE \left[  \tilde \bbP( \cA (d,f,\tau))^{\cc} \ | \ d ,f \in \tau \right] \le \gep,
\end{equation}
for arbitrary $\gep$.
Without loss of generality, let us assume that $d=0$. We set for notational simplicity $r:=\eta\ell/4$ and  
we define a new event $\cB (r,\tau)$ satisfying $\cB (r,\tau)\subset \cA (0,f,\tau)$ for all $f\ge =\eta\ell/4$,
\begin{equation}
\cB (r,\tau):= \left\{ \exists i \in \lint 1,r\rint, \exists j\in \lint 1,r^{\alpha/4} \rint \ : \ i\in \tau  , \ i+j\in \tau \text{ and }
\ \min(\go_i,\go_{i+j})\ge V(j) \right\}.
\end{equation}
Furthermore as a random variable in $\tau$, is is measurable with respect to $\tau \cap [0,f/2]$.
We want to use this assumption to drop the conditioning in $\tau$ present in \eqref{lapardeA}.
The reason to consider only dual peaks with relatively small distance ($\le r^{\alpha/4}$) is not of crucial importance but 
it notably simplifies the computation (cf.\ \eqref{lahousse}).

\begin{lemma}\label{abscont}
 There exists a constant such that for all $N>0$ for any  function $F$ measurable with respect to
 $\sigma(\tau \cap [0,N/2])$
we have 
\begin{equation}\label{fouquet}
 \bE[ F(\tau) \ | \ N \in \tau ]\le C  \bE[F(\tau)]
\end{equation}
\end{lemma}

\begin{proof}
If we let $X_N:=\max\{ \tau\cap [0,N/2]\}$, the left-hand side can be rewritten as 

\begin{multline}
\sum_{i=0}^{N/2} \bP[f(\tau) \ | \ X_i=N/2, N\in \tau] \bP[X_i=N/2 \ | \ N\in \tau]\\
=\sum_{i=0}^{N/2} \bP[f(\tau) \ | \ X_i=N/2] \bP[X_i=N/2 \ | \ N\in \tau],
\end{multline}
where the equality comes from the Markov property for the renewal $\tau$.
With this formulation, \eqref{fouquet} is simply a consequence of \cite[Equation (A.15)]{cf:DGLT07}.
 
\end{proof}
\noindent As a consequence of the Lemma we have 

\begin{equation}
\bE \left[  \tilde \bbP( \cA (0,f,\tau))^{\cc} \ | \ f \in \tau \right]\le \bE \left[  \tilde \bbP( \cB (r,\tau))^{\cc} \ | \ f \in \tau \right]
\le C \bE \left[  \tilde \bbP( \cB (r,\tau))^{\cc} \right]
\end{equation}

Hence to conclude the proof we only need to show the following which we leave to the next section.

\begin{lemma}\label{lapreuvebetau}
 Recall that $r=\eta \ell/4$. If $A$ is chosen sufficiently large (depending on $\eta$, $M$ and $\gep$), we have for all $\gb\in (0,1]$ 
  
  \begin{equation}
  \bE \left[  \tilde\bbP\left[ \cB (r,\tau)\right] \right]\ge 1-\gep.
 \end{equation}
 
\end{lemma}

\subsection{Proving Lemma \ref{lapreuvebetau}}

For $i,j,k$ in $\bbN$ let $\delta_{i}$, $\delta_{i,j}$ and $\delta_{i,j,k}$ be the indicator function of the respective events
$\{i \in \tau\}$, $\{i,i+j\in \tau\}$ and $\{i,i+j,i+j+k\in \tau\}$.
Using the independence of renewal jumps we have (recall the definition of $u(n)$ above \eqref{doney})
\begin{equation}\label{zouz}
 \bE[\delta_i]=u(i), \quad  \bE[\delta_{i,j}]=u(i)u(j), \ \text{ and } \ \bE[\delta_{i,j,k}]=u(i)u(j)u(k).
\end{equation}

Let us also set (recall \eqref{defV}) 
\begin{equation}
 W(i,j):=\ind_{\{\min(\go_i,\go_{i+j}\ge V(j)\}}
\end{equation}
We define
\begin{equation}
 Y(\go,\tau):=\sum_{i=1}^r \sum_{j=1}^{r^{\alpha/4}} W(i,j) \delta_{i,j}.
\end{equation}
With this notation the event $\cB (r,\tau)$ corresponds to $Y(\go,\tau)\ge 1$.
We are going to prove the lemma by controlling the two first moments of $Y$ w.r.t measure $\tilde \bbP$.

\medskip

We are going to use constantly the following estimates which can be deduced from the assumption \eqref{defgamma} the definition of the size 
biased measure and the value chosen for $\ell$: there exists a constant (depending on $M$) such that for every value of $i,j$ and $\gb\in(0,1]$ chosen we have (recall
$\alpha=1-\gamma^{-1}$) 
\begin{equation}\label{hola}
 (C_M)^{-1}\gb \left(\ell(\log \ell) n\right)^{-\frac{\alpha}{2}}   \le \tilde \bbP[  \go_1\ge V(n)]\le C_M\gb \left(\ell (\log \ell) n\right)^{-\frac{\alpha}{2}}  
\end{equation}

From now on, all the constant displayed in the equation might depend on $M$ and $\eta$ but not on other parameters.
Using \eqref{hola}  we have for some $c>0$
\begin{equation}
 \tilde\bbE[Y]\ge c \gb^2  (\ell \log \ell)^{-\alpha}\sum_{j=1}^{r^{\alpha/4}} j^{-\alpha} \delta_{i,j}.
\end{equation}
To compute the variance, we ignore after developing $Y^2=\sum_{i=1}^r \sum_{j=1}^{r^{\alpha/4}}\sum_{i'=1}^r \sum_{j'=1}^{r^{\alpha/4}}\dots$ all the terms which have covariance zero.
We are left with diagonal terms but also terms where $|\{i,i+j\} \cap \{i',i'+j'\}|=1$. Reordering the sum this gives the following estimate

\begin{multline}
  \var_{\tilde \bbP}[Y(\tau,\go)]\le \tilde \bbE \left[ \sum_{i=1}^r \sum_{j=1}^{r^{\alpha/4}} W(i,j) \delta_{i,j} \right]
  \\
  + 2 \tilde \bbE \left[ \sum_{i=1}^r \sum_{1\le j<k\le r^{\alpha/4}}  W(i,j)W(i,k) \delta_{i,j}\delta_{i,k} \right]\\
  + 2 \tilde \bbE \left[ \sum_{i=1}^r\sum_{j=1}^{r^{\alpha/4}}\sum_{k=1}^{r^{\alpha/4}}  W(i,j)W(i+j,k) \delta_{i,j,k} \right]
\end{multline}
Using \eqref{hola} to control all the expectation we obtain
\begin{multline}
 \var_{\tilde \bbP}[Y(\tau,\go)]\le C \gb^2  (\ell \log \ell)^{-\alpha} \sum_{i=1}^r \sum_{j=1}^{r^{\alpha/4}} j^{-\alpha} \delta_{i,j}\\
 + C \gb^3  (\ell \log \ell)^{-3\alpha/2} 
 \Bigg( \sum_{i=1}^r \sum_{1\le j<k\le r^{\alpha/4}} j^{-\alpha/2}k^{-\alpha}  \delta_{i,j,k-j}
         \\ + \sum_{i=1}^r\sum_{j=1}^{r^{\alpha/4}}\sum_{k=1}^{r^{\alpha/4}}  j^{-\alpha/2} k^{-\alpha/2} \max(j,k)^{-\alpha/2} \delta_{i,j,k} \Bigg).
 \end{multline}
 To conclude the proof of Lemma \ref{lapreuvebetau} we use the following estimates proved in the next section.
 
\begin{proposition}\label{thechosenone}
The following estimates hold for some universal constant $C$
\begin{itemize}
 \item [(i)] $\bE[\sum_{i=1}^r \sum_{j=1}^{r^{\alpha/4}} j^{-\alpha} \delta_{i,j}]\le C r^{\alpha}(\log r)$.
 \item [(ii)] There exists $\gep>0$ such that for all $r$ sufficiently large
 $$\bP\left[r^{-\alpha}(\log r)^{-1} \sum_{i=1}^r \sum_{j=1}^{r^{\alpha/4}} j^{-\alpha} \delta_{i,j} \ge \gep\right]\ge 1-\gep,$$
 \item [(iii)]
 \begin{equation} \begin{split}
 \bE\left[ \sum_{i=1}^r \sum_{1\le j<k\le r^{\alpha/4}} j^{-\alpha/2}k^{-\alpha}  \delta_{i,j,k-j} \right] & \le C r^{3\alpha/2}(\log r),\\
 \bE\left[ \sum_{i=1}^r\sum_{j=1}^{r^{\alpha/4}}\sum_{k=1}^{r^{\alpha/4}}  j^{-\alpha/2} k^{-\alpha/2} \max(j,k)^{-\alpha/2} \delta_{i,j,k} \right] 
 &\le C r^{3\alpha/2}(\log r).
  \end{split}
  \end{equation}
\end{itemize}

\end{proposition}

\medskip 

\noindent From $(ii)$ we obtain directly that provided $\gep$ is sufficiently small (how small can depend on $\eta$ and $M$) with probability larger than $(1-\gep)$ we have
\begin{equation}
 \tilde\bbE[Y(\tau,\go)]\ge \gep \gb^2 (\log \ell)^{\gamma^{-1}}
 \end{equation}
 From $(i)$ and $(iii)$, and Markov inequality we obtain that with probability larger than $1-\gep$ we have
 \begin{equation}
   \var_{\tilde \bbP}[Y(\go,\tau)]\le C(\eta,M)\gep^{-1}\left[  \gb^2 (\log \ell)^{\gamma^{-1}}+\gb^3(\log \ell)^{\gamma^{-1}-\alpha/2} \right].
 \end{equation}
 With our choice $\ell= e^{-A \gb^{2\gamma}}$ with probability larger than $(1-2\gep)$ we have 
 \begin{equation}\label{dsadsa}\begin{split}
   \tilde\bbE[Y(\tau,\go)]&\ge \gep A^{\gamma^{-1}},\\
      \var_{\tilde \bbP}[Y(\tau,\go)]&\le C(\eta,M)\gep^{-1} A^{\gamma^{-1}}+ O( \gb^{1+\alpha \gamma}).
      \end{split}
 \end{equation}
Thus by choosing $A$ sufficiently large (depending on $\eta$, $M$ and $\gep$), using Chebychev inequality we conclude that 

$$\bP\left[  Y(\tau,\go)\ge 1\right] \le 3 \gep.$$
 \qed 
 \subsection{Proof of Proposition \ref{thechosenone}}

We start with point $(i)$ and $(iii)$ which are simpler to prove.
Using \eqref{zouz} to rewrite the sum in $(i)$ and \eqref{doney} to estimate it, we obtain
 \begin{equation}
\sum_{i=1}^r \sum_{j=1}^{r^{\alpha/4}} u(i)u(j)j^{-\alpha} \le 
C \sum_{i=1}^r i^{\alpha-1} \sum_{j=1}^{r^{\alpha/4}} j^{-1}, 
\end{equation}
For $(iii)$ let us perform the computation only for the first sum since the other is similar.
 Using \eqref{zouz} and \eqref{doney} we have 
 \begin{multline}
  \sum_{i=1}^r \sum_{1\le j<k\le r^{\alpha/4}} u(i)u(j)u(k-j) j^{-\alpha/2}k^{-\alpha} 
  \\\le C\sum_{i=1}^r i^{\alpha-1} \sum_{j=1}^{r^{\alpha/4}} j^{-\alpha/2-1} \sum_{k=j+1}^{r^{\alpha/4}} (k-j)^{-1}
  \le C' r^{3\alpha/2}\log r.
  \end{multline}
Let us now consider the more delicate point $(ii)$.
 We set 
 \begin{equation}\begin{split}
 X^1_r&:= r^{-\alpha} \sum_{j=1}^{r} \delta_i,\\
 X^2_r&:= \left( \sum_{k=1}^{r^{\alpha/4}} k^{-\alpha} u(k) \right)^{-1} r^{-\alpha} \sum_{i=1}^{r} \sum_{j=1}^{r^{\alpha/4}} j^{-\alpha} \delta_{i,j}.
 \end{split}\end{equation}
  Note that as $\left( \sum_{k=1}^{r^{\alpha/4}} k^{-\alpha} u(k) \right)^{-1}$ is of order $\log r$, $X^2_r$ is asymptotically equivalent to the expression appearing in 
  $(ii)$. Hence it is sufficient to prove that 
  \begin{equation}
  \lim_{r\to \infty} \bP\left[X^2_r  \ge \gep\right]\ge 1-\gep.
  \end{equation}
We are going to show that $X^2_r$ converges in law and that the limit distribution does not give any mass to zero.
First we notice that as $n^{-1/\alpha}\tau_{\lceil n \rceil}$ converges to an $\alpha$ stable subordinator (see e.g. \cite[Chap. 16]{cf:Feller})
$X^1_r$ converges to the first hitting time of $[1,+\infty)$ for this process which is an almost surely positive random variable. 
Hence we conclude the proof using the following technical lemma, which readily implies that $X^2_r$ converges in distribution to the same random variable.
\begin{lemma}
 
 We have 
 
 \begin{equation}
  \lim_{r\to \infty} \bE\left[(X^1_r-X^2_r)^2\right]=0.
 \end{equation}

\end{lemma}

\begin{proof}

We have 
\begin{equation}
  r^{\alpha}  \left( \sum_{k=1}^{r^{\alpha/4}} k^{-\alpha} u(k) \right) (X^2_n-X^1_n)=
   \sum_{i=1}^{r}\delta_i \left( \sum_{j=1}^{r^{\alpha/4}} j^{-\alpha}\left(\delta_{i+j}-u(j)\right)\right)=:
     \sum_{i=1}^{r}U_i.
\end{equation}
Hence we have 
\begin{equation}\label{house}
 \bE \left[ (X^2_n-X^1_n)^2 \right]\le C r^{-2\alpha} (\log r)^{-2} \sum_{i_1,i_2=1}^r \bE[U_{i_1}U_{i_2}].
\end{equation}
We are going to show that we have  
\begin{equation}\label{lahousse}
\bE[U^2_{i}] \le C r^{\alpha/2}i^{\alpha-1}.
\end{equation}
and 
\begin{equation}\label{2lahousse}
|i_1-i_2|\ge r^{\alpha/4} \quad \Rightarrow \quad \bE[U_{i_1}U_{i_2}]=0, 
\end{equation}
Using these estimates we obtain that 
\begin{multline}
 \sum_{i_1,i_2=1}^r \bE[U_{i_1}U_{i_2}]\le \sum_{i=1}^r\bE[U^2_{i}]+ 2 \sum_{i_1=1}^r\sum_{i_2=(i_1+1)}^{\max(r,i_1+r^{\alpha/4})} \bE[U_{i_1}U_{i_2}]\\
 \le (1+2r^{\alpha/4})\sum_{i=1}^r\bE[U^2_{i}]\le C r^{3\alpha/4})\sum_{i=1}^r i^{\alpha-1}\le C r^{7\alpha/4}, 
 \end{multline}
which in regards of \eqref{house} allows to conclude.

\medskip

\noindent The inequality \eqref{lahousse} is simple to obtain. We have 
$$\left|\sum_{j=1}^{r^{\alpha/4}} j^{-\alpha}\left(\delta_{i+j}-u(j)\right)\right|\le r^{\alpha/4},$$
and hence 
$$\bE[U^2_{i}]\le r^{\alpha/2}\bE[\delta_i]\le  C  r^{\alpha/2} i^{\alpha-1}.$$
For \eqref{2lahousse}, we assume that $i_1$ is the smallest index.  Note that with the assumption $i_2-i_1\ge r^{\alpha/4}$, 
$U_{i_1}$ is measurable w.r.t.\ $\tau \cap  [0,i_2]$. Hence we have 
\begin{multline}
\bE[U_{i_1}U_{i_2} \ | \ \tau \cap  [0,i_2]]=
U_{i_1}\delta_{i_2} \sum_{j=1}^{r^{\alpha/4}} j^{-\alpha} \bE\left[ \delta_{i_2+j}-u(j) \ | \ \tau \cap  [0,i_2] \right]\\
= U_{i_1}\delta_{i_2}  \sum_{j=1}^{r^{\alpha/4}} j^{-\alpha}  \bE\left[\delta_{i_2+j}-u(j) \ | \ i_2\in \tau \right]=0.
\end{multline}
To obtain the second equality, we observe that both terms are equal to zero if  $i_2\notin \tau$ and that conditionally to $i_2\in \tau$,
$\tau\cap[0,i_2]$ and $\tau\cap[i_2,\infty)$ are independent.

\end{proof}

\end{document}